\documentclass[11pt, amsfonts]{amsart}

\usepackage{amsmath,amssymb,amsthm,enumerate,comment}
\usepackage[all]{xy}
\usepackage{color}
\usepackage{graphicx}
\usepackage[T1]{fontenc}
\usepackage[dvipdfm,colorlinks=true]{hyperref}
\usepackage[colorlinks=true]{hyperref}

\SelectTips{eu}{12}

\theoremstyle{definition}
\newtheorem{theorem}{Theorem}[section]

\newtheorem{lemma}[theorem]{Lemma}
\newtheorem{proposition}[theorem]{Proposition}
\newtheorem{corollary}[theorem]{Corollary}

\newtheorem{remark}[theorem]{\rm Remark}

\newtheorem{question}[theorem]{\rm Question}

\makeatletter
\@addtoreset{equation}{section} 
\makeatother

\DeclareMathOperator{\Ker}{Ker}

\DeclareMathOperator{\Mod}{Mod}

\DeclareMathOperator{\Homeo}{Homeo}

\newcommand{\RR}{\mathbb{R}}
\newcommand{\ZZ}{\mathbb{Z}}

\newcommand{\GG}{\Gamma}
\newcommand{\hGG}{\widehat{\GG}}
\newcommand{\Om}{\Omega}
\newcommand{\hOm}{\widehat{\Om}}
\renewcommand{\gg}{\gamma}

\newcommand{\h}{\Homeo_+(S^1)}
\renewcommand{\th}{\mathrm{H}\widetilde{\mathrm{omeo}}_+(S^1)}

\newcommand{\tf}{\widetilde{f}}
\newcommand{\tg}{\widetilde{g}}
\renewcommand{\Im}{\mathrm{Im}}
\newcommand{\trho}{\widetilde{\rho}}

\newcommand{\HHH}{\mathrm{H}}

\makeatletter
\@namedef{subjclassname@2020}{%
  \textup{2020} Mathematics Subject Classification}
\makeatother

\title{A crossed homomorphism for groups acting on the circle}
\author{Shuhei Maruyama}
\address{Graduate School of Mathematics, Nagoya University, Japan}
\email{m17037h@math.nagoya-u.ac.jp}


\begin{document}

\begin{abstract}
We construct a crossed homomorphism by using a group action on the circle and the Poincar\'{e} translation number.
We relate it to the Euler class of the action in terms of the Hochschild--Serre spectral sequence.
As an application, we answer a question of Calegari and Chen, which is on an explicit form of a certain crossed homomorphism on the mapping class group of the sphere minus a Cantor set.
\end{abstract}

\maketitle

\section{Introduction}
A group extension $1 \to K \xrightarrow{i} \GG \xrightarrow{p} G \to 1$ induces an exact sequence of cohomology groups
\begin{align*}
  0 &\to \HHH^1(G;\ZZ) \to \HHH^1(\GG;\ZZ) \to \HHH^1(K;\ZZ)^{\GG} \to \HHH^2(G;\ZZ) \\
  &\to \Ker\big(\HHH^2(\GG;\ZZ) \xrightarrow{i^*} \HHH^2(K;\ZZ)\big) \xrightarrow{\phi} \HHH^1(G;\HHH^1(K;\ZZ)) \xrightarrow{d_2^{1,1}} \HHH^3(G;\ZZ)
\end{align*}
called the seven-term exact sequence.
In this paper, we describe the image of the Euler class in $\Ker\big(\HHH^2(\GG;\ZZ) \xrightarrow{i^*} \HHH^2(K;\ZZ)\big)$ of a $\GG$-action on the circle under the map $\phi$ in terms of the Poincar\'{e} translation number.

Let $\h$ be the group of orientation preserving homeomorphisms of the circle and $\th$ its universal covering group.
It is well known that the second cohomology group $\HHH^2(\h;\ZZ)$ is isomorphic to $\ZZ$ and that the generator $e$ of $\HHH^2(\h;\ZZ)$ is called the \emph{(integral) Euler class of} $\h$.
For a homomorphism $\rho \colon \GG \to \h$, the pullback $\rho^*e \in \HHH^2(\GG;\ZZ)$ of the Euler class is called the \emph{Euler class of} $\rho$.

Assume that the Euler class $\rho^*e$ of $\rho$ is contained in $\Ker\big(\HHH^2(\GG;\ZZ) \xrightarrow{i^*} \HHH^2(K;\ZZ)\big)$.
In this paper, we provide a cocycle description of the class $\phi(\rho^*e) \in \HHH^1(G;\HHH^1(K;\ZZ))$.
Note that the cocycles of the class in $\HHH^1(G;\HHH^1(K;\ZZ))$ are just crossed homomorphisms $G \to \HHH^1(K;\ZZ)$ (see Subsection \ref{subsec:prel_crossed}).
Note also that, for the mapping class groups of closed hyperbolic surfaces with once-marked point, a similar construction of crossed homomorphisms was recently given by Chen \cite{2301.06247}, which uses the Dehn-Nielsen-Baer theorem (see Remark \ref{rem:Chen}).

The crossed homomorphism is given as follows.
Let $\rho \colon \GG \to \h$ be a homomorphism whose Euler class $\rho^*e$ is contained in $\Ker\big(\HHH^2(\GG;\ZZ) \xrightarrow{i^*} \HHH^2(K;\ZZ)\big)$.
By this assumption, the composite $\rho \circ i \colon K \to \h$ lifts to a homomorphism
\[
  \trho \colon K \to \th
\]
(see Proposition \ref{prop:lift_euler_class}).
Let $\tau \colon \th \to \RR$ be the Poincar\'{e} translation number.
We define a map $k'_{\trho} \colon \GG \to \HHH^1(K;\ZZ)$ by
\[
  k'_{\trho} (\gg) = \trho^*\tau - {}^{\gg}(\trho^*\tau)
\]
for every $\gg \in \GG$.
Here ${}^{\gg}(\trho^*\tau)$ is defined by
\[
  \big({}^{\gg}(\trho^*\tau)\big)(x) = \trho^*\tau(\gg^{-1} x \gg).
\]
for every $x \in K$.
We will see that the map $k'_{\trho} \colon \GG \to \HHH^1(K;\ZZ)$ is a well-defined crossed homomorphism (Lemmas \ref{lem:welldef} and \ref{lem:crossed}) and descends to $k_{\trho} \colon G \to \HHH^1(K;\ZZ)$ (Lemma \ref{lem:descend}).
Moreover, we will also see that the class $[k_{\trho}] \in \HHH^1(G;\HHH^1(K;\ZZ))$ does not depend on the choice of the lift $\trho$.

The main result in this paper is the following:
\begin{theorem}\label{thm:main_gen}
  The class $[k_{\trho}] \in \HHH^1(G;\HHH^1(K;\ZZ))$ represented by the crossed homomorphism $k_{\trho} \colon G \to \HHH^1(K;\ZZ)$ is equal to $\phi(\rho^*e)$.
  In particular, the class $[k_{\trho}] \in \HHH^1(G;\HHH^1(K;\ZZ))$ depends only on the semi-conjugacy class of $\rho$.
\end{theorem}

As an example, we consider the following big mapping class groups.
Let $C$ be a Cantor set in $S^2$ and set $\hOm = S^2 \setminus C$.
We take a point $\infty \in \hOm$, and set $\Om = \hOm \setminus \{ \infty \} = \RR^2 \setminus C$.
For the mapping class group $\hGG$ (resp. $\GG$) of $\hOm$ (resp. $\Om$), we have the following Birman exact sequence
\begin{align}\label{seq:birman}
  1 \to \pi_1(\hOm) \xrightarrow{i} \GG \xrightarrow{p} \hGG \to 1.
\end{align}
The seven-term exact sequence applied to (\ref{seq:birman}) induces an exact sequence
\begin{align}\label{seq:seven-term}
  0 \to \HHH^2(\GG;\ZZ) \xrightarrow{\phi} \HHH^1(\hGG;\HHH^1(\pi_1(\hOm);\ZZ)) \xrightarrow{d_2^{1,1}} \HHH^3(\hGG;\ZZ)
\end{align}
since $\HHH^2(\hGG;\ZZ) = 0$ (\cite{MR4266358}) and $\HHH^2(\pi_1(\hOm);\ZZ) = 0$.
It is shown in \cite{2211.07470} that $\HHH^2(\GG;\ZZ)$ is isomorphic to $\ZZ$.

Several non-trivial actions of $\GG$ on the circle $S^1$ has been constructed (\cite{Cal04}, \cite{MR3852444}, \cite{MR4266358}).
It was shown in \cite{MR4266358} that any non-trivial $\GG$-action on $S^1$ is semi-conjugate (up to a change of orientation) to the action $\rho_S$ on the simple circle constructed in \cite{MR4266358}.
In particular, the pullback of the integral Euler class $e \in \HHH^2(\h;\ZZ)$ by a non-trivial action $\rho \colon \GG \to \Homeo_+(S^1)$ is equal to either $\rho_S^*e$ or $-\rho_S^*e$.
We call this class $\rho_S^*e$ the \textit{Euler class of $\GG$}.
Related to the Euler class of $\GG$, Calegari and Chen raised in \cite{MR4266358} the following question.
\begin{question}[{\cite[Question A.16]{MR4266358}}]\label{q2}
  What class in $\HHH^1(\hGG;\HHH^1(\pi_1(\hOm);\ZZ))$ is the image of the Euler class in $\HHH^2(\GG;\ZZ)$?
\end{question}
By Theorem \ref{thm:main_gen} applied to (\ref{seq:birman}) and the action $\rho_{S}$, the crossed homomorphism $k_{\widetilde{\rho_S}}$ provides an answer to the question.
In this case, the injectivity of $\phi$ implies the following.

\begin{corollary}
  The class $[k_{\trho}] \in \HHH^1(\hGG;\HHH^1(\pi_1(\hOm);\ZZ))$ is non-zero if $\rho \colon G \to \h$ is non-trivial.
\end{corollary}

\section{Preliminaries}

\subsection{Group cohomology and crossed homomorphisms}\label{subsec:prel_crossed}
In this subsection, we review the definition of group cohomology and crossed homomorphism.
We refer the reader to \cite{brown82} for details.

Let $G$ be a group and $M$ a left $G$-module.
The $G$-action on $M$ will be denoted by ${}^{g}x$ for every $g \in G$ and $x \in M$.
We set $C^n(G;M) = \{ c \colon G^n \to M \}$ for every $n > 0$ and $C^0(G;M) = M$.
An element of $C^n(G;M)$ is called an \textit{$n$-cochain on $G$ with coefficients in $M$}.
The coboundary map $\delta \colon C^n(G;M) \to C^{n+1}(G;M)$ is defined by
\begin{align*}
  &\delta c (g_1, \cdots, g_{n+1}) \\
  & = {}^{g_1}\big( c(g_2, \cdots, g_{n+1}) \big) + \sum_{i = 1}^{n} (-1)^{i}c(g_1, \cdots, g_i g_{i+1}, \cdots, g_{n+1}) + (-1)^{n+1} c(g_1, \cdots, g_{n})
\end{align*}
for every $c \in C^n(G;M)$ and $g_1, \cdots, g_{n+1} \in G$.
The homology of the cochain complex $(C^*(G;M), \delta)$ is called the \textit{group cohomology of $G$ with coefficients in $M$} and denoted by $\HHH^*(G;M)$.

A map $k \colon G \to M$ is called a \textit{crossed homomorphism} if the equality
\[
  k(g_1 g_2) = k(g_1) + {}^{g_1}\big(k(g_2) \big)
\]
holds for every $g_1, g_2 \in G$.
By definition, a map $k \colon G \to M$ is a crossed homomorphism if and only if $k$ is a cocycle in $C^1(G;M)$.
Hence, any first cohomology class in $\HHH^1(G;M)$ is represented by a crossed homomorphism.
Note that, if $M$ is a trivial $G$-module, a crossed homomorphism is just a homomorphism from $G$ to $M$.

A crossed homomorphism relating to mapping class groups (of finite-type surfaces) has been constructed (\cite{MR1030850}, \cite{MR1617632} and references therein; also \cite{2301.06247}).
Our crossed homomorphism is similar to one introduced recently in \cite{2301.06247}.

\subsection{The Poincar\'{e} translation number and the Euler class}

Let us recall the definition of the Poincar\'{e} translation number and Euler cocycles induced from it.
Let $T \colon \RR \to \RR$ be the translation by one.
One of the descriptions of the universal covering group $\th$ is given by
\[
  \th = \{ f \in \Homeo(\RR) \colon f \circ T = T \circ f \}.
\]
Under this description, the Poincar\'{e} translation number $\tau \colon \th \to \RR$ is defined by
\[
  \tau(f) = \lim_{n \to \infty} \frac{f^n(0)}{n}.
\]
By the definition of $\tau$ and the fact that $T$ is in the center of $\th$, we have
\[
  \tau(f \circ T^m) = \tau(f) + m
\]
for every $m \in \ZZ$.
Note that the Poincar\'{e} translation number is invariant under conjugation, that is, the equality
\[
  \tau(f_2^{-1} f_1 f_2) = \tau(f_1)
\]
holds for every $f_1, f_2 \in \th$ (see \cite[Proposition 5.3]{MR1876932}).

For every $f,g \in \h$ and their lifts $\tf, \tg $, we set
\[
  \chi(f,g) = \tau(\tf \tg) - \tau(\tf) - \tau(\tg) \in \RR.
\]
This $\chi(f,g)$ does not depend on the choice of the lifts.
Moreover, it turns out that the cochain $\chi \in C^2(\h;\RR)$ is a cocycle representing the real Euler class $e_{\RR} \in \HHH^2(\h;\RR)$.
This cocycle $\chi$ was introduced by Matsumoto \cite{MR848896}, and is called the \textit{canonical Euler cocycle}.

In the proof of Theorem \ref{thm:main_gen}, we will use a variant of the canonical Euler cocycle defined as follows:
Let $\lfloor \cdot \rfloor$ denote the floor function.
For every $f,g \in \h$ and their lifts $\tf, \tg $, we set
\begin{align}\label{int_euler_cocycle}
  \chi_{\ZZ}(f,g) = \lfloor\tau(\tf \tg)\rfloor - \lfloor\tau(\tf)\rfloor - \lfloor\tau(\tg)\rfloor \in \ZZ.
\end{align}
This $\chi_{\ZZ}(f,g)$ is also independent of the choice of the lifts.
This integer-valued cocycle $\chi_{\ZZ} \in C^2(\h;\ZZ)$ represents the \textit{integral} Euler class $e \in \HHH^2(\h;\ZZ)$.
Indeed, the function $\widetilde{b} \colon \th \to \RR$ defined by
\[
  \widetilde{b}(\tf) = \tau(\tf) - \lfloor \tau(\tf) \rfloor
\]
descends to a function $b \colon \h \to \RR$, and hence $\chi = \chi_{\ZZ} + \delta b \in C^2(\h;\RR)$.
This implies that $\chi_{\ZZ}$ represents the \textit{real} Euler class in $\HHH^2(\h;\RR)$.
Since the change of coefficients homomorphism $\HHH^2(\h;\ZZ) \to \HHH^2(\h;\RR)$ is injective, the cocycle $\chi_{\ZZ}$ represents the integral Euler class.

The following fact on the integral Euler class is used in the construction of the crossed homomorphism in the introduction.
\begin{proposition}[{see \cite[Subsection 6.2]{MR1876932}}]\label{prop:lift_euler_class}
  Let $G$ be a group and $\rho \colon G \to \h$ a homomorphism.
  The homomorphism $\rho$ lifts to a homomorphism $\trho \colon G \to \th$ if and only if the Euler class $\rho^* e \in \HHH^2(G;\ZZ)$ of $\rho$ is eqaul to zero.
\end{proposition}

\section{On the crossed homomorphism}

Let $1 \to K \xrightarrow{i} \GG \xrightarrow{p} G \to 1$ be a group extension.
Recall that the map $k'_{\trho} \colon \GG \to \HHH^1(K;\ZZ)$ is defined by
\[
  k'_{\trho}(\gg) = \trho^*\tau - {}^{\gg}(\trho^*\tau)
\]
for every $\gg \in \GG$, where $\trho \colon K \to \th$ is a lift of the composition of $i \colon K \to \GG$ and $\rho \colon \GG \to \h$.

\begin{lemma}\label{lem:welldef}
  The map $k'_{\trho} \colon \GG \to \HHH^1(K;\ZZ)$ is well-defined.
\end{lemma}
\begin{proof}
  First we show that the map $k'_{\trho}(\gg) \colon K \to \RR$ is integer-valued for every $\gg \in \GG$.
  We take $\gg \in \GG$ and $x \in K$.
  By the equality
  \[
    \pi(\trho (\gg^{-1} x \gg)) = \rho(\gg^{-1} x \gg) = \rho(\gg)^{-1} \rho(x) \rho(\gg),
  \]
  there exists an integer $n$ such that
  \[
    \trho(\gg^{-1} x \gg) = \widetilde{\rho(\gg)}^{-1} \trho(x) \widetilde{\rho(\gg)} \circ T^{-n},
  \]
  where $\widetilde{\rho(\gg)} \in \th$ is a lift of $\rho(\gg) \in \h$.
  Note that the integer $n$ is independent of the choice of the lift $\widetilde{\rho(\gg)}$ since $T$ is in the center of $\th$.
  Hence we have
  \begin{align}\label{calc_k}
    \big(k'_{\trho}(\gg) \big)(x) &= \tau(\trho(x)) - \tau(\trho(\gg^{-1}x\gg))\\
    &= \tau(\trho(x)) - \tau(\widetilde{\rho(\gg)}^{-1} \trho(x) \widetilde{\rho(\gg)} \circ T^{-n}) \nonumber \\
    &= \tau(\trho(x)) - \big(\tau(\widetilde{\rho(\gg)}^{-1} \trho(x) \widetilde{\rho(\gg)}) - n \big) \nonumber \\
    &= \tau(\trho(x)) - \tau(\trho(x)) + n \nonumber \\
    &= n \nonumber
  \end{align}
  by the conjugation-invariance of $\tau$.
  Hence the map $k'_{\trho}(\gg)$ is integer-valued for every $\gg \in \GG$.

  To see that the map $k'_{\trho}(\gg) \colon K \to \ZZ$ is a homomorphism, we take $x, y \in K$.
  By the above argument, there exist integers $n$ and $m$ such that
  \begin{align*}
    \trho(\gg^{-1} x \gg) = \widetilde{\rho(\gg)}^{-1} \trho(x) \widetilde{\rho(\gg)} \circ T^{-n}, \\
    \trho(\gg^{-1} y \gg) = \widetilde{\rho(\gg)}^{-1} \trho(y) \widetilde{\rho(\gg)} \circ T^{-m}.
  \end{align*}
  Then, we have
  \begin{align*}
    \trho(\gg^{-1} xy \gg) = \widetilde{\rho(\gg)}^{-1} \trho(xy) \widetilde{\rho(\gg)} \circ T^{-n-m}.
  \end{align*}
  By the calculation same as in (\ref{calc_k}), we have
  \[
    \big(k'_{\trho}(\gg)\big)(xy) = n+m = \big(k'_{\trho}(\gg)\big)(x) + \big(k'_{\trho}(\gg)\big)(y).
  \]
  Hence the map $k'_{\trho}(\gg)$ is a homomorphism.
\end{proof}

\begin{lemma}\label{lem:crossed}
  The map $k'_{\trho} \colon \GG \to \HHH^1(K;\ZZ)$ is a crossed homomorphism with respect to the left $\GG$-action on $\HHH^1(K;\ZZ)$ induced from the conjugation.
\end{lemma}

\begin{proof}
  For every $\gg_1, \gg_2 \in \GG$, we have
  \[
    k'_{\trho}(\gg_1 \gg_2) = \trho^* \tau - {}^{\gg_1 \gg_2}(\trho^* \tau)
    = \trho^* \tau - {}^{\gg_1} (\trho^* \tau) + {}^{\gg_1} (\trho^* \tau) - {}^{\gg_1 \gg_2}(\trho^* \tau) = k'_{\trho}(\gg_1) + {}^{\gg_1} k'_{\trho}(\gg_2).
  \]
\end{proof}

\begin{remark}
  The crossed homomorphism $k'_{\trho}$ depends on the choice of $\trho$.
  However, the class $[k'_{\trho}]$ in $\HHH^1(\GG;\HHH^1(K;\ZZ))$ \textit{does not} depend on the choice of $\trho$.
  Indeed, we take two lifts $\trho_1$ and $\trho_2$ of $\rho$.
  Then we have
  \[
    \delta (\trho_1^* \tau - \trho_2^* \tau) = \trho_1^* \delta \tau - \trho^* \delta \tau = -(\trho_1^* \pi^* \chi + \trho_2^* \pi^* \chi) = -(i^* \rho^* \chi - i^* \rho^* \chi) = 0.
  \]
  Hence $\trho_1^* \tau - \trho_2^* \tau$ is a homomorphism.
  Moreover, since $\trho_1$ and $\trho_2$ are lifts of the same homomorphism $i \circ \rho$, the homomorphism $\trho_1^* \tau - \trho_2^* \tau$ is integer-valued.
  Therefore we have
  \begin{align*}
    (k'_{\trho_1} - k'_{\trho_2})(\gg) = \trho_1^* \tau - \trho_2^* \tau - {}^{\gg}(\trho_1^* \tau - \trho_2^* \tau),
  \end{align*}
  and this implies that the class $[k'_{\trho_1} - k'_{\trho_2}]$ is equal to zero in $\HHH^1(\GG;\HHH^1(K;\ZZ))$.
\end{remark}

\begin{lemma}\label{lem:descend}
  The map $k'_{\trho} \colon \GG \to \HHH^1(K;\ZZ)$ is trivial on $K$.
  In particular, the map $k'_{\trho} \colon \GG \to \HHH^1(K;\ZZ)$ descends to a map $k_{\trho} \colon G \to \HHH^1(K;\ZZ)$.
\end{lemma}

\begin{proof}
  For every $x, y \in K \subset \GG$, we have
  \begin{align*}
    \big(k'_{\trho}(x)\big)(y) &= \trho^* \tau (y) - \big({}^{x}(\trho^* \tau)\big)(y) = \tau(\trho(y)) - \tau(\trho(x)^{-1}\trho(y)\trho(x))\\
    &= \tau(\trho(y)) - \tau(\trho(y)) = 0
  \end{align*}
  by the conjugation invariance of $\tau$.
  Hence $k'_{\trho}$ is trivial on $K$.
  In particular, we obtain
  \[
    k'_{\trho}(\gg x) = k'_{\trho}(\gg) + {}^{\gg} (k'_{\trho}(x)) = k'_{\trho}(\gg),
  \]
  this implies the lemma.
\end{proof}

\begin{remark}\label{rem:Chen}
  Let $S_g$ be a closed connected oriented surface of genus $g \geq 2$ and
  $\Mod_{g}$ (resp. $\Mod_{g,*}$) the mapping class group of $S_g$ (resp. the mapping class group of $S_g$ with once-marked point).
  In \cite{2301.06247}, Chen constructed a crossed homomorphism $\Mod_{g,*} \to \HHH^1(S_g;\ZZ)$ as follows:
  Let
  \[
    1 \to \pi_1(S_g) \xrightarrow{i} \Mod_{g,*} \to \Mod_g \to 1
  \]
  be the Birman exact sequence and $\rho \colon \Mod_{g,*} \to \h$ the Gromov boundary action.
  Contrasting to the above case, the restricted action $\rho \circ i \colon \pi_1(S_g) \to \h$ \emph{does not} lift to $\pi_1(S_g) \to \h$.
  However, for the natural surjection $q \colon F_{2g} \to \pi_1(S_g)$ from the free group $F_{2g} = \langle a_1, b_1, \cdots, a_g, b_g \rangle$ of rank $2g$, the composite $\rho \circ i \circ q \colon F_{2g} \to \h$ \emph{lifts} to $\trho \colon F_{2g} \to \th$.
  For every $\gg \in \Mod_{g,*}$, we take a representative automorphism $f_{\gg} \colon F_{2g} \to F_{2g}$ which fixes the element $[a_1, b_1] \cdots [a_g,b_g] \in F_{2g}$ (the existence is guaranteed by the Dehn-Nielsen-Baer theorem).
  Then the crossed homomorphism $R \colon \Mod_{g,*} \to \HHH^1(F_{2g};\ZZ) \cong \HHH^1(S_g;\ZZ)$ is defined by
  \[
    R(\gg) = \big((\trho^* \tau)\circ f_{\gg}\big) - \trho^* \tau.
  \]
  It is also shown in \cite{2301.06247} that the crossed homomorphism $k$ enjoys the property
  \[
    \big( R(x) \big)(y) = (2g-2) i(x,y)
  \]
  for every $x,y \in \pi_1(S_g)$ and their algebraic intersection number $i(x,y)$, which ensures that the crossed homomorphism $R$ gives the generator of $\HHH^1(\Mod_{g,*};\HHH^1(S_g;\ZZ)) \cong \ZZ$.
  This is contrasting to Lemma \ref{lem:descend}, and this discrepancy is based on the difference of the liftability of the actions $\rho \circ i \colon \pi_1(S_g) \to \h$ and $\rho \circ i \colon K \to \h$ to $\th$;
  $\rho \circ i \colon \pi_1(S_g) \to \h$ is not liftable but $\rho \circ i \colon K \to \h$ is liftable.
\end{remark}

\section{Proof of Theorem \ref{thm:main_gen}}
To prove Theorem \ref{thm:main_gen}, we recall the Hochschild-Serre spectral sequence.
Let
\begin{align}\label{seq:ex}
  1 \to K \xrightarrow{i} \GG \xrightarrow{p} G \to 1
\end{align}
be an exact sequence of groups and $M$ a $\GG$-module.
The Hochschild-Serre spectral sequence $(E_r^{p,q}, d_r^{p,q})$ is a first quadrant spectral sequence converging to $\HHH^*(\GG;M)$ with
\begin{align}\label{isom_e2}
  E_2^{p,q} \cong \HHH^p(G;\HHH^q(K;M)).
\end{align}
This spectral sequence is given as follows.

An $n$-cochain $c \in C^n(\GG;M)$ is called to be \textit{normalized} if
\[
  c(\gg_1, \cdots, \gg_n) = 0
\]
whenever $\gg_i = 1_{\GG}$ for some $i$.
Let $A^n(\GG;M)$ be the group of normalized $n$-cochains on $\GG$.
Note that the cohomology of the subcomplex $(A^*(\GG;M), \delta)$ is isomorphic to $\HHH^*(\GG;M)$.
We set
\begin{align*}
  A_j^{*n} = \{ f \in A^n(\GG;M) \colon f(\gg_1, \cdots, \gg_n) & \text{ depends only on } \\
  &\gg_1, \cdots, \gg_{n-j} \in \GG \text{ and } p(\gg_{n-j+1}), \cdots, p(\gg_n) \in G \}.
\end{align*}
This defines a decreasing filtration
\[
  A^n(\GG;M) = A_0^{*n} \supset A_1^{*n} \supset \cdots \supset A_n^{*n} \supset A_{n+1}^{*n} = 0.
\]
This filtration induces the Hochschild-Serre spectral sequence.
More precisely, $E_r^{p,q}$ is given by
\[
  E_r^{p,q} = Z_r^{p, q} / (Z_{r-1}^{p+1, q-1} + \delta A_{p-r+1}^{*p+q-1}),
\]
where $Z_r^{p,q}=\{ f \in A_{p}^{*p+q} \colon \delta f \in A_{p+r}^{*p+q-1} \}$, and isomorphism (\ref{isom_e2}) is induced from the map
\[
  r_p\colon Z_2^{p,q} \to C^p(G;C^q(K;M))
\]
defined by
\[
  r_p(f)(p(\gg_1), \cdots, p(\gg_p)) = f(*, \cdots, *, \gg_1, \cdots, \gg_p)
\]
(see \cite[Theorems 1 and 2]{hochschild_serre53}).

A first quadrant spectral sequence $(E_r^{p,q}, d_r^{p,q})$ which converges to $E^*$ induces an exact sequence
\[
  0 \to E_2^{0,1} \to E^1 \to E_2^{0,1} \to E_2^{2,0} \to F_1E^2 \to E_2^{1,1} \to E_2^{3,0},
\]
which is called the \textit{seven-term exact sequence}.
Here the map $F_1E^2 \to E_2^{1,1}$ is induced from the natural projection
\[
  F_1E^2 \to F_1E^2/F_2E^2 = E_{\infty}^{1,1} \subset E_2^{1,1}.
\]
When $(E_r^{p,q}, d_r^{p,q})$ is the Hochschild-Serre spectral sequecne, the filtration $F_pE^{p+q}$ is given by
\[
  F_pE^{p+q} = \HHH^{p+q}(\Im (A_{p}^{*\bullet} \to A^{*\bullet})),
\]
and it is known that $F_1E^{2} \cong \Ker \big(\HHH^2(\GG;M) \to \HHH^2(K;M)\big)$.
Hence the seven-term exact sequence for the Hochschild-Serre spectral sequence gives rise to
\begin{align*}
  0 &\to \HHH^1(G;\ZZ) \to \HHH^1(\GG;\ZZ) \to \HHH^1(K;\ZZ)^{\GG} \to \HHH^2(G;\ZZ) \\
  &\to \Ker\big(\HHH^2(\GG;\ZZ) \xrightarrow{i^*} \HHH^2(K;\ZZ)\big) \xrightarrow{\phi} \HHH^1(G;\HHH^1(K;\ZZ)) \xrightarrow{d_2^{1,1}} \HHH^3(G;\ZZ),
\end{align*}
and the image of a class $c \in \Ker\big(\HHH^2(\GG;M) \to \HHH^2(K;M)\big)$ under the map
\[
  \phi \colon \Ker\big(\HHH^2(\GG;M) \to \HHH^2(K;M)\big) \to \HHH^1(G;\HHH^1(K;M))
\]
is given as follows:
We take a cocycle $f \in A_1^{*2}$ representing $c$, and define a crossed homomorphism $\mathfrak{K}_f \colon G \to \HHH^1(K;M)$ by
\[
  \big(\mathfrak{K}_f(p(\gg))\big) (x) = f(x, \gg)
\]
for every $\gg \in \GG$ and $x \in K$.
Then, by definition, we obtain $\phi(c) = [\mathfrak{K}_f]$.


\begin{proof}[Proof of Theorem $\ref{thm:main_gen}$]
  Let $\chi_{\ZZ} \in C^2(\h;\ZZ)$ be the Euler cocycle defined by (\ref{int_euler_cocycle}).
  We identify $K$ with $i(K) \subset \GG$ through the map $i$ and define a function $a_{\trho} \colon \GG \to \ZZ$ by
  \begin{align*}
    a_{\trho}(\gg) =
    \begin{cases}
      \lfloor\tau(\trho(\gg))\rfloor & \gg \in K \\
      0 & \text{otherwise.}
    \end{cases}
  \end{align*}
  We take a section $s \colon G \to \GG$ satisfying $s(1_{G}) = 1_{\GG}$ and set
  \[
    x_{\gg} = s(p(\gg))^{-1} \gg \in K.
  \]
  Let us define a function $b_{\trho} \colon \GG \to \ZZ$ by
  \begin{align*}
    b_{\trho}(\gg) =
    \begin{cases}
      \lfloor\tau(\trho(\gg))\rfloor & \gg \in K \\
      (\rho^*\chi_{\ZZ} + \delta a_{\trho})(s(p(\gg)), x_{\gg}) & \text{otherwise}.
    \end{cases}
  \end{align*}
  Then, the cocycle $\rho^* \chi_{\ZZ} + \delta b_{\trho}$ represents the Euler class $\rho^*e$ and is contained in $A_1^{*2}$.
  To see $\rho^* \chi_{\ZZ} + \delta b_{\trho} \in A_1^{*2}$, we take $\gg \in \GG$ and $x, y \in K$.
  If $\gg \notin K$, we have
  \begin{align} \label{eq1}
    \delta b_{\trho}(s(p(\gg)) x, y) =& b_{\trho}(y) - b_{\trho}(s(p(\gg)) xy) + b_{\trho}(s(p(\gg)) x) \\
    =& \lfloor\tau(\trho(y))\rfloor - \rho^*\chi_{\ZZ}(s(p(\gg)), xy) - \delta a_{\trho}(s(p(\gg)), xy)\nonumber \\
    & +\rho^*\chi_{\ZZ}(s(p(\gg)), x) + \delta a_{\trho}(s(p(\gg)), x)\nonumber \\
    =& \lfloor\tau(\trho(y))\rfloor - \lfloor\tau(\trho(xy))\rfloor + \lfloor\tau(\trho(x))\rfloor - \rho^*\chi_{\ZZ}(s(p(\gg))x, y) + \rho^*\chi_{\ZZ}(x,y)\nonumber \\
    =& - \rho^*\chi_{\ZZ}(s(p(\gg))x, y),\nonumber
  \end{align}
  where the third equality is from $\delta \rho^*\chi_{\ZZ} = 0$ and the fourth equality is from the definition of $\chi_{\ZZ}$.
  If $\gg \in K$, we obtain
  \begin{align} \label{eq2}
    \delta b_{\trho}(s(p(\gg)) x, y) =& \delta b_{\trho}(x, y) = \lfloor \tau(\trho(y)) \rfloor - \lfloor \tau(\trho(xy)) \rfloor + \lfloor \tau(\trho(x)) \rfloor\\
    =& -\rho^* \chi_{\ZZ}(x, y) = -\rho^* \chi_{\ZZ}(s(p(\gg)) x, y). \nonumber
  \end{align}
  Equalities (\ref{eq1}) and (\ref{eq2}) imply that
  \begin{align}\label{eq3}
    (\rho^*\chi_{\ZZ} + \delta b_{\trho})(\gg, y) = 0
  \end{align}
  for every $\gg \in \GG$ and $y \in K$.
  Hence, for every $\gg_1, \gg_2 \in \GG$ and $y \in K$, we obtain
  \begin{align*}
    (\rho^*\chi_{\ZZ} + \delta b_{\trho})(\gg_1, \gg_2 y) &= -(\rho^*\chi_{\ZZ} + \delta b_{\trho})(\gg_2, y) + (\rho^*\chi_{\ZZ} + \delta b_{\trho})(\gg_1 \gg_2, y) + (\rho^*\chi_{\ZZ} + \delta b_{\trho})(\gg_1, \gg_2) \\
    &= (\rho^*\chi_{\ZZ} + \delta b_{\trho})(\gg_1, \gg_2).
  \end{align*}
  This implies that $\rho^*\chi_{\ZZ} + \delta b_{\trho} \in A_{1}^{*2}$.

  By the above description of $\phi$, the class $\phi(\rho^* e)$ is represented by the crossed homomorphism $\mathfrak{K}_{\rho^*\chi + \delta b_{\trho}}$.
  By the definition of $\mathfrak{K}_{\rho^*\chi + \delta b_{\trho}}$ and (\ref{eq3}), we have
  \begin{align*}
    \big(\mathfrak{K}_{\rho^*\chi_{\ZZ} + \delta b_{\trho}}(p(\gg))\big)(x) = (\rho^*\chi_{\ZZ} + \delta b_{\trho})(x, \gg) = (\rho^*\chi_{\ZZ} + \delta b_{\trho})(x, \gg) - (\rho^*\chi_{\ZZ} + \delta b_{\trho})(\gg,\gg^{-1}x\gg)
  \end{align*}
  for every $\gg \in \GG$ and $x \in K$.
  By the definition of the canonical Euler cocycle $\chi$ and the conjugation invariance of $\tau$, we have
  \begin{align*}
    &\rho^* \chi_{\ZZ}(x, \gg) + \rho^*\chi_{\ZZ}(\gg, \gg^{-1}x\gg) \\
    &= \lfloor\tau(\widetilde{\rho(x)}\widetilde{\rho(\gg)})\rfloor - \lfloor\tau(\widetilde{\rho(x)})\rfloor - \lfloor\tau(\widetilde{\rho(\gg)})\rfloor \\
    &\hspace{5mm} -\left(\lfloor\tau(\widetilde{\rho(\gg)} \cdot \widetilde{\rho(\gg)}^{-1}\widetilde{\rho(x)}\widetilde{\rho(\gg)})\rfloor - \lfloor\tau(\widetilde{\rho(\gg)})\rfloor - \lfloor\tau(\widetilde{\rho(\gg)}^{-1}\widetilde{\rho(x)}\widetilde{\rho(\gg)})\rfloor\right) \\
    &= \lfloor\tau(\widetilde{\rho(x)}\widetilde{\rho(\gg)})\rfloor - \lfloor\tau(\widetilde{\rho(x)})\rfloor - \lfloor\tau(\widetilde{\rho(\gg)})\rfloor \\
    &\hspace{5mm} -\left(\lfloor\tau(\widetilde{\rho(x)}\widetilde{\rho(\gg)})\rfloor - \lfloor\tau(\widetilde{\rho(\gg)})\rfloor - \lfloor\tau(\widetilde{\rho(x)})\rfloor\right) \\
    &= 0.
  \end{align*}
  Moreover, by the definition of $b_{\trho}$, we obtain
  \begin{align*}
    \delta b_{\trho}(x, \gg) - \delta b_{\trho}(\gg, \gg^{-1}x\gg) =& b_{\trho}(\gg) - b_{\trho}(x\gg) + b_{\trho}(x) - \big( b_{\trho}(\gg^{-1}x\gg) - b_{\trho}(x\gg) + b_{\trho}(\gg) \big)\\
    =& b_{\trho}(x) - b_{\trho}(\gg^{-1}x\gg) = \lfloor\tau(\trho(x))\rfloor - \lfloor\tau(\trho(\gg^{-1}x\gg))\rfloor.
  \end{align*}
  Since $\tau(\trho(x)) - \tau(\trho(\gg^{-1}x\gg)) \in \ZZ$, we have
  \[
    \lfloor\tau(\trho(x))\rfloor - \lfloor\tau(\trho(\gg^{-1}x\gg))\rfloor = \tau(\trho(x)) - \tau(\trho(\gg^{-1}x\gg)) = \big(k_{\trho}(\gg)\big)(x).
  \]
  Hence the equality $\mathfrak{K}_{\rho^*\chi + \delta b_{\trho}} = k_{\trho}$ holds.
  This implies that
  \[
    \phi(\rho^* e) = [\mathfrak{K}_{\rho^*\chi + \delta b_{\trho}}] = [k_{\trho}] \in \HHH^1(G, \HHH^1(K;\ZZ)).
  \]
\end{proof}

\section*{Acknowledgments}
The author would like to thank Lei Chen, Lvzhou Chen, Teruaki Kitano, Takayuki Morifuji, Shigeyuki Morita and Takuya Sakasai for helpful comments.
The author is supported by JSPS KAKENHI Grant Number JP21J11199.

\bibliographystyle{amsalpha}
\bibliography{references.bib}
\end{document}